\documentclass[graybox]{svmult}

\usepackage{type1cm}        
\usepackage{makeidx}         
\usepackage{graphicx}        
\usepackage{multicol}        
\usepackage[bottom]{footmisc}

\usepackage{newtxtext}       %
\usepackage[varvw]{newtxmath}       

\usepackage[colorlinks=true, allcolors=blue]{hyperref}
\usepackage{tikz}
\usepackage{wrapfig}

\newcommand{\norm}[1]{\left\lVert#1\right\rVert}

\makeindex             

\begin{document}

\title*{From PDEs constrained optimization to controllability problems via time domain decomposition}
\titlerunning{From PDEs constrained optimization to controllability problems}

\author{Pierre-Henri Cocquet\orcidID{0000-0003-0986-0800} and Liu-Di Lu\orcidID{0000-0003-3629-9879}}
\institute{
Pierre-Henri Cocquet \at Université de Pau et des Pays de l'Adour, France, \email{pierre-henri.cocquet@univ-pau.fr}
\and Liu-Di Lu \at Lund University, Sweden, \email{liudi.lu@math.lu.se}
}
%
%
\maketitle

\abstract*{This paper focuses on the application of time domain decomposition to solve partial differential equations constrained optimization problems and controllability problems. After clarifying the link between these two types of problems, we show that applying time domain decomposition to both problems leads to the same convergence behavior. Our numerical experiments also confirm these theoretical findings.}

\section{Introduction}\label{sec:1}
Although partial differential equations (PDEs) constrained optimization and controllability problems share similarities, they originated from different historical contexts. Controllability problems can be traced back more than 300 years to Johann Bernoulli’s brachystochrone problem; see~\cite{Sussmann1997} for a historical review. This problem initiated the development of optimal control theory, in which the system is often governed by ordinary differential equations (ODEs). PDE-constrained optimization, on the other hand, was developed in the early 1960s by the desire to extend control theory to systems governed by PDEs. This field has become an active research field following the monograph of Jacques-Louis Lions~\cite{Lions1971}.

In this paper, we focus on applying domain decomposition techniques to solve both classes of problems. To illustrate these two problems, we consider the linear heat equation as the governing constraint in both cases. In the case of a PDE-constrained optimization problem, we have a given target $\hat y$ and want to find optimal control strategies $u$ such that the corresponding solution $y$ of the heat equation is as close as possible to the given target. This leads to a constrained minimization problem:
\begin{equation}\label{eq:P1}
\begin{aligned}
\min_{y, u} \ \frac{\alpha}{2} &\norm{y -\hat y}^2_{L^2(Q)} 
+ \frac{\gamma}{2} \norm{y(T, \cdot) - \hat y(T, \cdot)}^2_{L^2(\Omega)} 
+ \frac{\nu}{2} \norm{u}^2_{L^2(Q)},
\\
\text{s.t.} \quad
&\partial_t y - \kappa \Delta y = u \ \text{ in } Q := (0, T)\times\Omega,
\\
&y = 0 \ \text{ on } \Sigma := (0, T)\times\partial\Omega
\quad
y = y_0 \ \text{ on } \Sigma_0 := \{0\}\times\Omega.
\end{aligned}
\tag{P1}
\end{equation}
Here, $\alpha \geq 0, \gamma \geq 0$ and $\nu > 0$ are given penalization parameters, $\kappa>0$ is the diffusion coefficient which is assumed to be constant in the spatial domain, $y_0$ is a given initial condition, $T>0$ is the final time and $\Omega\subset \mathbb{R}^n$ with $n = 1, 2, 3$ is a bounded open set with Lipschitz boundary. For our controllability problem, we want to find optimal control strategies $u$ with minimal $L^2$-norm that pilot the solution of the heat equation $y$ from a given initial condition $y_0$ to a zero final condition. This results in solving a similar minimization problem to~\eqref{eq:P1} as
\begin{equation}\label{eq:P2}
\begin{aligned}
&\min_{u} \ \frac{\nu}{2}\norm{u}^2_{L^2(Q)},
\\
\text{s.t.} \quad
\partial_t y - \kappa \Delta y = u \text{ in } Q, 
\
&y|_{\Sigma} = 0, 
\
y = y_0 \text{ on } \Sigma_0, 
\
y = 0 \text{ on } \Sigma_T:= \{T\}\times\Omega.
\end{aligned}
\tag{P2}
\end{equation}
The existence and uniqueness of a solution to~\eqref{eq:P1} are guaranteed by the strict-convexity of the linear quadratic optimization problem. The null-controllability of the classical heat equation is also well known, {\it e.g.}, see~\cite[Chapter 2.5]{coron2007control}. The existence of solution $u\in L^2(Q)$ to our null-controllability problem~\eqref{eq:P2} has been proven in~\cite[Theorem 6.1.2]{puel2019control} for $\Omega$ with $\mathcal{C}^{2+\delta}$ boundary, $\delta>0$. We emphasize that the results of~\cite[Theorem 6.1.2]{puel2019control} have been obtained considering a weak solution to the heat equation.
\begin{wrapfigure}{r}{0.3\textwidth}
\begin{tikzpicture}
\node at (0, 2) (a){\eqref{eq:P1}};
\node at (3, 2) (b){\eqref{eq:P2}};
\node at (3, 0) (c){$\rho_{\text{P2}}$};
\node at (0, 0) (d){$\rho_{\text{P1}}$};
\draw [thick, ->] (a) -- (b) node[midway, above] {$\alpha = 0$, $\gamma \to \infty$};
\draw [thick, ->] (b) -- (c) node[midway, left] {DD};
\draw [thick, ->] (d) -- (c) node[midway, below] {$\alpha = 0$, $\gamma \to \infty$};
\draw [thick, ->] (a) -- (d) node[midway, right] {DD};
\end{tikzpicture}
\end{wrapfigure}
The similarity between these two problems~\eqref{eq:P1} and~\eqref{eq:P2} already appears in their formulations. Indeed, choosing $\hat y(T, \cdot) = 0$ and setting $\alpha=0$, $\gamma\to+\infty$ in~\eqref{eq:P1} show that~\eqref{eq:P1} \textit{"converges"} toward~\eqref{eq:P2}. We make this convergence more rigorous in Section~\ref{sec:2}. On the other hand, if we apply separately time domain decomposition to the first order optimality system issued from each problem, we show in Section~\ref{sec:3} that under the same conditions, the convergence factor associated with~\eqref{eq:P1}, $\rho_{\text{P1}}$, coincide with the convergence factor associated with~\eqref{eq:P2}, $\rho_{\text{P2}}$. This relation is illustrated with the commuting diagram, and our numerical experiments in Section~\ref{sec:4} confirm these results.

\section{Relation between problems (P1) and (P2)}\label{sec:2}

First of all, we introduce $y_u$ as the unique weak solution to 
\begin{equation}\label{eq:direct_heat_PDE}
\partial_t y_{u} - \kappa \Delta y_{u} = u \ \text{ in } (0, T)\times\Omega,
\
y_u = 0 \ \text{ on } \Sigma,
\
y_u = y_0 \ \text{ on } \Sigma_0.
\end{equation}
A weak solution to~\eqref{eq:direct_heat_PDE} is defined as $y_u\in  L^2(0,T;H^1_0(\Omega))$, $\partial_t y_u\in L^2(0,T;H^{-1}(\Omega))$ and $y_u|_{\Sigma_0}=y_0$ in $L^2(\Omega)$ satisfying for all $\varphi\in H^1_0(\Omega)$,
\[\left<\partial_t y_u, \varphi\right>_{H^{-1},H^1_0} + \int_{\Omega}\kappa\nabla y_u\cdot\nabla\varphi\,dx =  \int_{\Omega}u\varphi\, dx, \ \text{a.e.}\ t\in]0,T[.\]
Aubin's Lemma (see~\cite[Lemma 6.2]{ern2004theory}) ensures that any weak solution $y_u\in \mathcal{C}([0,T],L^2(\Omega))$, hence the initial condition $y_0\in L^2(\Omega)$ is meaningful. From~\cite[Theorem 6.6]{ern2004theory}, there exists a unique weak solution to~\eqref{eq:direct_heat_PDE} which satisfies
\begin{equation}\label{eq:Bound_direct_heat}
\begin{aligned}
\norm{\partial_t y_u}_{L^2(0,T;H^{-1}(\Omega))}+\norm{y_u}_{L^\infty(0,T;L^2(\Omega))} &+ \norm{y_u}_{L^2(0,T;H^1_{0}(\Omega)} \\
&\leq C \left(\norm{u}_{L^2(Q)}+\norm{y_0}_{L^2(\Omega)}\right),
\end{aligned}
\end{equation}
for a suitable constant $C>0$. For the relation between~\eqref{eq:P1} and~\eqref{eq:P2}, we have the next result (see~\cite[Theorem 10]{ciaramella2015newton} for a similar result in a finite-dimensional case).

\begin{theorem}
Let $y_T:=0$ and $u^*$ be a solution of~\eqref{eq:P2}. Let $u_\varepsilon^*$ be the unique solution of~\eqref{eq:P1} with $\alpha=0$, $\gamma = 1/\varepsilon$ and $\hat{y}(T, \cdot)=y_T$ whose existence and uniqueness follows from the strict-convexity of \eqref{eq:P1}. Let $y_{\varepsilon}^* = y_{u_\varepsilon^*}$ then 
\begin{equation}\label{eq:Bound_P1_P2}
\norm{y_\varepsilon^*(T, \cdot)-y_T}_{L^2(\Omega)} \leq \sqrt{\nu\varepsilon}\norm{u^*}_{L^2(Q)}.
\end{equation}
In addition, there exist accumulation points of $(u_\varepsilon^*)_\varepsilon\subset L^2(Q)$ for the weak topology of $L^2(Q)$, and any of such accumulation points is an optimal solution to~\eqref{eq:P2}. If~\eqref{eq:P2} has the unique solution $u^*$, then the whole sequence $(u_\varepsilon^*)_\varepsilon$ converges toward $u^*$ strongly in $L^2(Q)$.
\end{theorem}

\begin{proof}
The optimality of $u_\varepsilon^*$ ensures that $\forall u\in L^2(Q)$, one has
\[\frac{1}{2\varepsilon} \norm{y_\varepsilon^*(T, \cdot) - y_T}^2_{L^2(\Omega)} 
+ \frac{\nu}{2} \norm{u_\varepsilon^*}^2_{L^2(Q)} \leq 
\frac{1}{2\varepsilon} \norm{y_u(T, \cdot) - y_T}^2_{L^2(\Omega)} 
+ \frac{\nu}{2} \norm{u}^2_{L^2(Q)}.\]
Now taking $u=u^*$ and using that $y_{u^*}$ satisfies $y_{u^*}(T,\cdot)=y_T$ give 
\[\frac{1}{2\varepsilon} \norm{y_\varepsilon^*(T, \cdot) - y_T}^2_{L^2(\Omega)} 
+ \frac{\nu}{2} \norm{u_\varepsilon^*}^2_{L^2(Q)} \leq 
\frac{\nu}{2} \norm{u^*}^2_{L^2(Q)},\]
from which the aforementioned inequality follows. 
We also get that the sequence $\left( u_\varepsilon^*\right)_\varepsilon$ is uniformly bounded in $L^2(Q)$ and thus admits a subsequence that converges weakly in $L^2(Q)$ toward some $\widetilde{u}$. From~\eqref{eq:Bound_direct_heat}, the sequence $\left( y_{\varepsilon}^*\right)_{\varepsilon}$ is uniformly bounded in some reflexive Banach space $\mathcal{H}$ and thus admits a subsequence that weakly converges toward some $\widetilde{y}\in \mathcal{H}$. Taking the limit $\varepsilon\to 0$ in the weak formulation of~\eqref{eq:direct_heat_PDE} and using that $u_\varepsilon^*\to\widetilde{u}$, we obtain (up to a subsequence) that $y_\varepsilon^*\to \widetilde{y} = y_{\widetilde{u}}$. Owing to~\eqref{eq:Bound_P1_P2}, $y_{\varepsilon}^*(T, \cdot)\to y_T$ strongly in $L^2(\Omega)$ so that $y_{\widetilde{u}}(T,\cdot) = y_T$. The weak lower semicontinuity of the $L^2-$norm finally yields 
\[\norm{\widetilde{u}}_{L^2(Q)} \leq \liminf_{\varepsilon\to 0} \norm{u_\varepsilon^*}_{L^2(Q)}\leq \norm{u^*}_{L^2(Q)},\] 
and the optimality of $u^*$ ensures that $\norm{\widetilde{u}}_{L^2(Q)} =\norm{u^*}_{L^2(Q)}$. Hence, $\widetilde{u}$ is an optimal solution to~\eqref{eq:P2}. Any accumulation point (for the weak topology of $L^2(Q)$) of $\left(u_\varepsilon \right)_\varepsilon$ is then an optimal solution to~\eqref{eq:P2}. If~\eqref{eq:P2} has a unique optimal solution $u^*$, then $\widetilde{u}=u^*$ and the whole sequence $(u_\varepsilon)_\varepsilon$ converges toward $u^*$. The uniqueness of the limit gives that 
\[\liminf_{\varepsilon\to 0} \norm{u_\varepsilon^*}_{L^2(Q)}= \lim_{\varepsilon\to 0}\norm{u_\varepsilon^*}_{L^2(Q)} =  \norm{u^*}_{L^2(Q)},\] 
and then $u^*_\varepsilon\to u^*$ strongly.
\end{proof}

\section{Time domain decomposition}\label{sec:3}

A classical way to treat these two problems~\eqref{eq:P1} and~\eqref{eq:P2} is to derive the first-order optimality system by applying the Lagrange multiplier approach. We can then obtain the reduced optimality system for~\eqref{eq:P1},
\begin{equation}\label{eq:S1}
\begin{aligned}
\partial_t y - \kappa \Delta y &= \nu^{-1}\lambda &&\text{ in } Q, 
&&y = y_0 && \text{ on } \Sigma_0, \\
\partial_t \lambda + \kappa \Delta \lambda &= \alpha (y - \hat y) && \text{ in } Q,
&&\lambda = -\gamma (y - \hat y) && \text{ on } \Sigma_T,
\end{aligned}
\tag{S1}
\end{equation}
and the reduced optimality system for~\eqref{eq:P2},
\begin{equation}\label{eq:S2}
\begin{aligned}
\partial_t y - \kappa \Delta y &= \nu^{-1}\lambda && \text{ in } Q, 
\quad
y = y_0 \ \text{ on } \Sigma_0,
\quad
y = 0 \ \text{ on } \Sigma_T, \\
\partial_t \lambda + \kappa \Delta \lambda &= 0 && \text{ in } Q, 
\end{aligned}
\tag{S2}
\end{equation}
where $\lambda$ is the adjoint state. We refer to~\cite[Chapter 3 Section 3.1]{Lions1971} for further details on the derivation of the reduced optimality system. We used the optimality condition $\nu u - \lambda = 0$ in $Q$ in both cases to eliminate the control variable $u$ in the state equation of $y$, and we omit the Dirichlet boundary conditions in~\eqref{eq:S1} and~\eqref{eq:S2}. We also assume that both problems~\eqref{eq:P1}, \eqref{eq:P2} are solvable. We are now interested in applying domain decomposition methods to solve their associated optimality systems~\eqref{eq:S1} and~\eqref{eq:S2}.

The alternating Schwarz algorithm, invented by H.A. Schwarz in 1869, is a way to solve a large problem by decomposing the original domain into smaller subdomains and then solving the problem on each subdomain one after another, while exchanging information between interfaces until the solution converges. We refer to~\cite{Gander2008} for a historical review. It is well known that the alternating Schwarz algorithm fails to converge when applied to non-overlapping spatial subdomains. This is due to the repeated passing of identical "Dirichlet data" from one subdomain to the other.  However, it has been shown in~\cite{Gander2025} that there exist several variants of the alternating Schwarz algorithm in the time decomposition framework, and some of them may converge even without overlap. When decomposing in time and applying non-overlapping alternating Schwarz type algorithms to the reduced optimality system, a Dirichlet transmission condition at the time interface can be equivalently rewritten as a Robin type transmission condition, e.g., see~\cite[Equation 4]{Gander2025}. This avoids passing identical "Dirichlet data" between subdomains and thus leads to the convergence. For this reason, we are interested here in studying the non-overlapping time-decomposed alternating Schwarz algorithms to solve~\eqref{eq:S1} and~\eqref{eq:S2}. We decompose the space-time domain $Q$ into two non-overlapping subdomains: $Q_1 := (0,\Gamma)\times\Omega$ and $Q_2 := (\Gamma, T)\times\Omega$, where $\Gamma\in (0, T)$ represents the interface.

The system~\eqref{eq:S1} has a forward backward structure, where the state variable $y$ propagates forward with an initial condition $y_0$, and the adjoint variable $\lambda$ propagates backward with a final condition $-\gamma(y(T, \cdot) - \hat y(T, \cdot))$. To retain the same forward-backward structure in subdomain $Q_1$, one should impose a "final" condition at the interface $\Sigma_{\Gamma} := \{\Gamma\}\times\Omega$ for the adjoint variable $\lambda$. Similarly, one imposes an "initial" condition at the interface $\Sigma_{\Gamma}$ for the state variable $y$ to have the same forward-backward structure in $Q_2$. We denote by AS1 this first type of alternating Schwarz algorithm which imposes $\lambda$ at $\Sigma_{\Gamma}$ in $Q_1$ and $y$ at $\Sigma_{\Gamma}$ in $Q_2$. Then, for the iteration index $\ell = 1, 2, \ldots$, the AS1 applied to solve~\eqref{eq:S1} reads
\begin{equation}\label{eq:S1AS1}
\begin{aligned}
\partial_t y_1^{\ell} - \kappa \Delta y_1^{\ell} &= \nu^{-1}\lambda_1^{\ell}&& \text{ in } Q_1, \\
y_1^{\ell} &= y_0 && \text{ on } \Sigma_0 , \\
\partial_t \lambda_1^{\ell} + \kappa \Delta \lambda_1^{\ell} &= \alpha (y_1^{\ell} - \hat y) && \text{ in } Q_1, \\
\lambda_1^{\ell} &= \lambda_2^{\ell-1} && \text{ on } \Sigma_{\Gamma},
\end{aligned}
\
\begin{aligned}
\partial_t y_2^{\ell} - \kappa \Delta y_2^{\ell} &= \nu^{-1}\lambda_2^{\ell} && \text{ in } Q_2, \\
y_2^{\ell} &= y_1^{\ell} && \text{ on } \Sigma_{\Gamma}, \\
\partial_t \lambda_2^{\ell} + \kappa \Delta \lambda_2^{\ell} &= \alpha (y_2^{\ell} - \hat y) && \text{ in } Q_2, \\
\lambda_2^{\ell} + \gamma y_2^{\ell} &= \gamma \hat y && \text{ on } \Sigma_T.
\end{aligned}
\tag{S1-AS1}
\end{equation}
For comparison, we also apply AS1 to solve~\eqref{eq:S2}, which leads to the algorithm,
\begin{equation}\label{eq:S2AS1}
\begin{aligned}
\partial_t y_1^{\ell} - \kappa \Delta y_1^{\ell} &= \nu^{-1}\lambda_1^{\ell} && \text{ in } Q_1, \\
y_1^{\ell} &= y_0 && \text{ on } \Sigma_0 , \\
\partial_t \lambda_1^{\ell} + \kappa \Delta \lambda_1^{\ell} &= 0 && \text{ in } Q_1, \\
\lambda_1^{\ell} &= \lambda_2^{\ell-1} && \text{ on } \Sigma_{\Gamma},
\end{aligned}
\qquad
\begin{aligned}
\partial_t y_2^{\ell} - \kappa \Delta y_2^{\ell} &= \nu^{-1}\lambda_2^{\ell} && \text{ in } Q_2, \\
y_2^{\ell} &= y_1^{\ell} && \text{ on } \Sigma_{\Gamma}, \\
y_2^{\ell} &= 0  && \text{ on } \Sigma_T\\
\partial_t \lambda_2^{\ell} + \kappa \Delta \lambda_2^{\ell} &= 0 && \text{ in } Q_2.\\
\end{aligned}
\tag{S2-AS1}
\end{equation}
We observe that in the first subdomain $Q_1$ on the left of the system~\eqref{eq:S2AS1}, it does not follow the same structure as in~\eqref{eq:S2}. Indeed, the system~\eqref{eq:S2} is also a forward-backward system in which the initial and final conditions are all given for the state variable $y$. Hence, to retain the same forward backward structure of~\eqref{eq:S2} when applying the alternating Schwarz algorithm, one should impose a "final" condition at the interface $\Sigma_{\Gamma}$ for the state variable $y$ in $Q_1$ and an "initial" condition at the interface $\Sigma_{\Gamma}$ also for $y$ in $Q_2$. This then leads to a second type of alternating Schwarz algorithm denoted by AS2, which imposes $y$ at $\Sigma_{\Gamma}$ in both $Q_1$ and $Q_2$. The AS2 applied to solve~\eqref{eq:S2} then reads
\begin{equation}\label{eq:S2AS2}
\begin{aligned}
\partial_t y_1^{\ell} - \kappa \Delta y_1^{\ell} &= \nu^{-1}\lambda_1^{\ell} && \text{ in } Q_1, \\
y_1^{\ell} &= y_0 && \text{ on } \Sigma_0 , \\
y_1^{\ell} &= y_2^{\ell-1} && \text{ on } \Sigma_{\Gamma}, \\
\partial_t \lambda_1^{\ell} + \kappa \Delta \lambda_1^{\ell} &= 0 && \text{ in } Q_1, 
\end{aligned}
\qquad
\begin{aligned}
\partial_t y_2^{\ell} - \kappa \Delta y_2^{\ell} &= \nu^{-1}\lambda_2^{\ell} && \text{ in } Q_2, \\
y_2^{\ell} &= y_1^{\ell} && \text{ on } \Sigma_{\Gamma}, \\
y_2^{\ell} &= 0  && \text{ on } \Sigma_T\\
\partial_t \lambda_2^{\ell} + \kappa \Delta \lambda_2^{\ell} &= 0 && \text{ in } Q_2,\\
\end{aligned}
\tag{S2-AS2}
\end{equation}
which now follows the same structure as~\eqref{eq:S2} in both subdomains. For comparison, we also apply AS2 to solve~\eqref{eq:S1}, which leads to the algorithm,
 \begin{equation}\label{eq:S1AS2}
\begin{aligned}
\partial_t y_1^{\ell} - \kappa \Delta y_1^{\ell} &= \nu^{-1}\lambda_1^{\ell} && \text{ in } Q_1, \\
y_1^{\ell} &= y_0 && \text{ on } \Sigma_0 , \\
y_1^{\ell} &= y_2^{\ell-1} && \text{ on } \Sigma_{\Gamma}, \\
\partial_t \lambda_1^{\ell} + \kappa \Delta \lambda_1^{\ell} &= \alpha (y_1^{\ell} - \hat y) && \text{ in } Q_1, \\
\end{aligned}
\
\begin{aligned}
\partial_t y_2^{\ell} - \kappa \Delta y_2^{\ell} &= \nu^{-1}\lambda_2^{\ell} && \text{ in } Q_2, \\
y_2^{\ell} &= y_1^{\ell} && \text{ on } \Sigma_{\Gamma}, \\
\partial_t \lambda_2^{\ell} + \kappa \Delta \lambda_2^{\ell} &= \alpha (y_2^{\ell} - \hat y) && \text{ in } Q_2, \\
\lambda_2^{\ell} + \gamma y_2^{\ell} &= \gamma \hat y && \text{ on } \Sigma_T.
\end{aligned}
\tag{S1-AS2}
\end{equation}

To get better insights into these four algorithms~\eqref{eq:S1AS1}-\eqref{eq:S1AS2}, we need to analyze their convergence. Since we focus here on the time domain decomposition, we can then apply a discretization in space and replace the spatial operator $-\Delta$ by a matrix $A\in\mathbb{R}^{N \times N}$, e.g., using the centered finite difference discretization in space. As we only have Dirichlet boundary conditions, this finite difference matrix $A$ is symmetric and can be diagonalized with $PP^T = I_N$, $PAP^T = D := \text{diag}(d_1, \ldots, d_N)$, and $d_i > 0$, $i = 1, \ldots, N$. For instance, the eigenvalues in the one dimensional case ($n=1$) are given by $d_i = \frac{4}{h^2}\sin^2(\frac{\pi i}{2(N+1)})$, $i = 1, \ldots, N$, with $h:=L/(N+1)$ the mesh size and $L$ the length of the spatial domain. This diagonalization leads to $N$ independent $2 \times 2$ systems of first-order ODEs in time.

For each algorithm, solving the ODE system explicitly for each eigenvalue $d_i$ first in $Q_1$ then in $Q_2$ leads to a recursive relation between two iterates $\boldsymbol{y}_{2, i}^{\ell} = \rho(d_i)\boldsymbol{y}_{2, i}^{\ell-1}$, where $\boldsymbol{y}_{2, i}^{\ell}(t) \approx y_2^{\ell}(t, x_i)$ is an approximation after the discretization and diagonalization in space, and $\rho(d_i)$ is the contraction factor associated with the eigenvalue $d_i$. Each algorithm in~\eqref{eq:S1AS1}-\eqref{eq:S2AS1} has an associated contraction factor $\rho(d_i)$ for each eigenvalue $d_i$, as will be shown in the sequel. The absolute value of the factor $|\rho(d_i)|$ measures the contraction of each algorithm as a function of $d_i$, and the algorithm converges if this absolute value is smaller than one for all eigenvalues, {\it i.e.}, $\max_{d_i} |\rho(d_i)| < 1$. More details of the diagonalization and computations to obtain the contraction factor can be found in~\cite{Gander2025}.

The contraction factor associated with~\eqref{eq:S1AS1} for each $d_i$ is given by
\begin{equation}\label{eq:rhoS1AS1}
\rho^{\text{AS1}}_{\text{S1}}(d_i) = 
\frac{\alpha + \gamma(\sigma_i\coth\left(\sigma_i (T-\Gamma)\right) - \kappa d_i)}
{\nu \big(\sigma_i\coth(\sigma_i \Gamma) + \kappa d_i\big)
\big(\sigma_i\coth\left(\sigma_i (T-\Gamma)\right) + \kappa d_i + \gamma\nu^{-1}\big)},
\end{equation}
with $\sigma_i := \sqrt{\kappa^2d_i^2 + \nu^{-1}\alpha}$. Similarly, the contraction factor of~\eqref{eq:S2AS1} for each eigenvalue $d_i$ is given by
\begin{equation}\label{eq:rhoS2AS1}
\rho^{\text{AS1}}_{\text{S2}}(d_i) = \frac{\coth\left(\kappa d_i (T-\Gamma)\right) - 1}{\coth(\kappa d_i \Gamma) + 1}.
\end{equation}
Setting $\alpha = 0$ in~\eqref{eq:rhoS1AS1}, we have
\begin{equation}\label{eq:rhoS1AS1alpha0}
\rho^{\text{AS1}}_{\text{S1}}(d_i)\Big|_{\alpha=0} =  \frac{\coth\left(\kappa d_i (T-\Gamma)\right) - 1}{\coth(\kappa d_i \Gamma) + 1}
\frac{\gamma}{\nu\kappa d_i\coth\left(\kappa d_i (T-\Gamma)\right) + \nu\kappa d_i + \gamma}.
\end{equation}
Then, when $\gamma \to \infty$, the second fraction goes to 1, thus 
\[\lim_{\gamma \to \infty} \rho^{\text{AS1}}_{\text{S1}}(d_i)\Big|_{\alpha=0}  = \rho^{\text{AS1}}_{\text{S2}}(d_i),\] 
meaning that these two contraction factors become the same in the limit case. Since this limit holds for each eigenvalue, we have the next result.

\begin{theorem}
If $\alpha=0$ in~\eqref{eq:P1}, then the convergence behavior of two algorithms~\eqref{eq:S1AS1} and~\eqref{eq:S2AS1} becomes the same when $\gamma \to \infty$.
\end{theorem}

In terms of two problems~\eqref{eq:P1} and~\eqref{eq:P2}, when $\alpha = 0$, or more generally, taking a relatively small value of $\alpha$ w.r.t. $\gamma$ and $\nu$ in~\eqref{eq:P1}, it means that the discrepancy between the solution $y$ of the heat equation and the given target $\hat y$ is less important. In this context, taking large values of $\gamma$ emphasizes that one searches for optimal control strategies in~\eqref{eq:P1} that can pilot the solution of the heat equation to the target only at the final time $T$. This is equivalent to the idea of the controllability problem~\eqref{eq:P2}. In this case, it is relevant that we also obtain similar convergence behavior when applying the same type of alternating Schwarz algorithm to solve the corresponding optimality system. In particular, their convergence does not depend on the choice of the target function $\hat y$ in~\eqref{eq:P1}, nor on the final condition $y_T$ in~\eqref{eq:P2}.

In terms of two algorithms~\eqref{eq:S1AS1} and~\eqref{eq:S2AS1}, it has been shown in~\cite[Theorem 1]{Gander2025} that the algorithm~\eqref{eq:S1AS1} always converges for $\Gamma\leq T/2$ when $\alpha=1$. This still holds for the algorithm~\eqref{eq:S2AS1}. Indeed, using two properties of the hyperbolic cotangent: $\coth(x) > 1$ and decreases for all $x>0$, we find in~\eqref{eq:rhoS2AS1} that $\coth(\kappa d_i (T-\Gamma)) \leq \coth(\kappa d_i \Gamma)$. Thus, $\rho^{\text{AS1}}_{\text{S2}}(d_i) < 1$ for each $d_i$, and the algorithm~\eqref{eq:S2AS1} converges for $\Gamma\leq T/2$. Given the simple analytical form of~\eqref{eq:rhoS2AS1}, we can be even more precise about the choice of $\Gamma$.

\begin{theorem}\label{thm:Gamma}
Assume that $\Gamma < \frac{1}{2\kappa d_{\min}}\ln{\frac{1}{2}(1+e^{2\kappa d_{\min}T})}$ where $d_{\min}$ is the smallest eigenvalue of $A$. Then $\max_{d_i}|\rho^{\text{AS1}}_{\text{S2}}(d_i)|<1$ hence Algorithm~\eqref{eq:S2AS1} converges.
\end{theorem}

\begin{proof}
It suffices to study the contraction factor~\eqref{eq:rhoS2AS1} for each $d_i$. We can first remove the absolute value in~\eqref{eq:rhoS2AS1}, since $\coth(x) > 1$ for all $x>0$. Then $\rho^{\text{AS1}}_{\text{S2}}(d_i)<1$ is equivalent to $\coth(\kappa d_i (T-\Gamma)) - \coth(\kappa d_i \Gamma)  < 2$. Note that for $x\in(0, 1)$, the equation $\coth(a(1-x)) - \coth(ax) < 2$ leads to the solution $x < f(a)$ with 
\[f(a):=\frac{1}{2a}\ln{\frac{1}{2}(1+e^{2a})}.\] 
In particular, the function $f$ is strictly positive and increasing for $a>0$. Substituting $a$ by $\kappa d_iT$ and $x$ by $\Gamma/T$ in the inequality $x < f(a)$, we then obtain that 
\[\frac{\Gamma}T < \frac{1}{2\kappa d_iT}\ln{\frac{1}{2}(1+e^{2\kappa d_iT})}.\] 
This inequality needs to be satisfied for all eigenvalues $d_i$, $i = 1, \ldots, N$. Hence, we obtain $\Gamma < \frac{1}{2\kappa d_{\min}}\ln{\frac{1}{2}(1+e^{2\kappa d_{\min}T})}$.
\end{proof}

In the one dimensional case with only Dirichlet boundary conditions, the smallest eigenvalue is 
\[d_{\min} = \frac{4}{h^2}\sin^2(\frac{\pi}{2(N+1)})= \frac{\pi^2}{L^2} - h^2\frac{\pi^4}{12 L^4} + O(h^4).\]
Note that for the function $f$, $\lim_{a\to 0} f(a) = 1/2$ and $\lim_{a\to \infty} f(a) = 1$. We thus have the bound 
\[\frac{T}{2} < \frac{T}{2\kappa d_{\min}T}\ln{\frac{1}{2}(1+e^{2\kappa d_{\min}T})} < T,\] 
which also holds in higher dimensions, with eigenvalues $d_i\in(0, \infty)$, $i = 1, \ldots, N$. On the other hand, the contraction factor~\eqref{eq:rhoS2AS1} decreases as a function of $d_i$, which is bounded above by $\rho^{\text{AS1}}_{\text{S2}}(d_{\min})$. The convergence behavior then only depends on the smallest eigenvalue.

Following the same approach, we can solve the ODE systems associated with~\eqref{eq:S2AS2} and~\eqref{eq:S1AS2} to find their contraction factors. For each $d_i$, these two contraction factors are
\[\rho^{\text{AS2}}_{\text{S1}}(d_i) = 1,\quad
\rho^{\text{AS2}}_{\text{S2}}(d_i) = 1.\] 
This is not surprising because of the transmission conditions used in~\eqref{eq:S2AS2} and~\eqref{eq:S1AS2}. For both algorithms, we have $y_1^{\ell} = y_2^{\ell-1}$ on $\Sigma_{\Gamma}$ in $Q_1$, and $y_2^{\ell} = y_1^{\ell}$ on $\Sigma_{\Gamma}$ in $Q_2$. For any given initial guess $y_2^0$, the same information is constantly passed back and forth between two subdomains, solving the same problem at each iteration. Thus, both algorithms stagnate. This does not depend on the eigenvalues or on any parameter values, since these transmission conditions are imposed to decompose~\eqref{eq:S1} and~\eqref{eq:S2} at the continuous level. Recall that these transmission conditions are the natural choice to decompose the optimality system~\eqref{eq:S2}, since we maintain the same structure in both subdomains in~\eqref{eq:S2AS2} as in the original optimality system~\eqref{eq:S2}. This reveals once again the importance of choosing properly the transmission condition in the non-overlapping alternating Schwarz in time framework as discussed in~\cite{Gander2025}.

\section{Numerical experiments}\label{sec:4}

To illustrate the convergence behavior of these alternating Schwarz algorithms, we first plot the contraction factors~\eqref{eq:rhoS1AS1} and~\eqref{eq:rhoS2AS1} as functions of eigenvalues. We set $\alpha = 1$, $\nu = 1$, $\kappa = 1$, $T = 1$ and an equal decomposition with $\Gamma = T/2$. Fig.~\ref{fig:illustration} on the top left illustrates two contraction factors~\eqref{eq:rhoS1AS1} and~\eqref{eq:rhoS2AS1} for eigenvalues $d_i\in[10^{-2}, 10^2]$ and different values of $\gamma\in\{1, 10, 10^2\}$. 

\begin{figure}
\centering
\includegraphics[scale = 0.13]{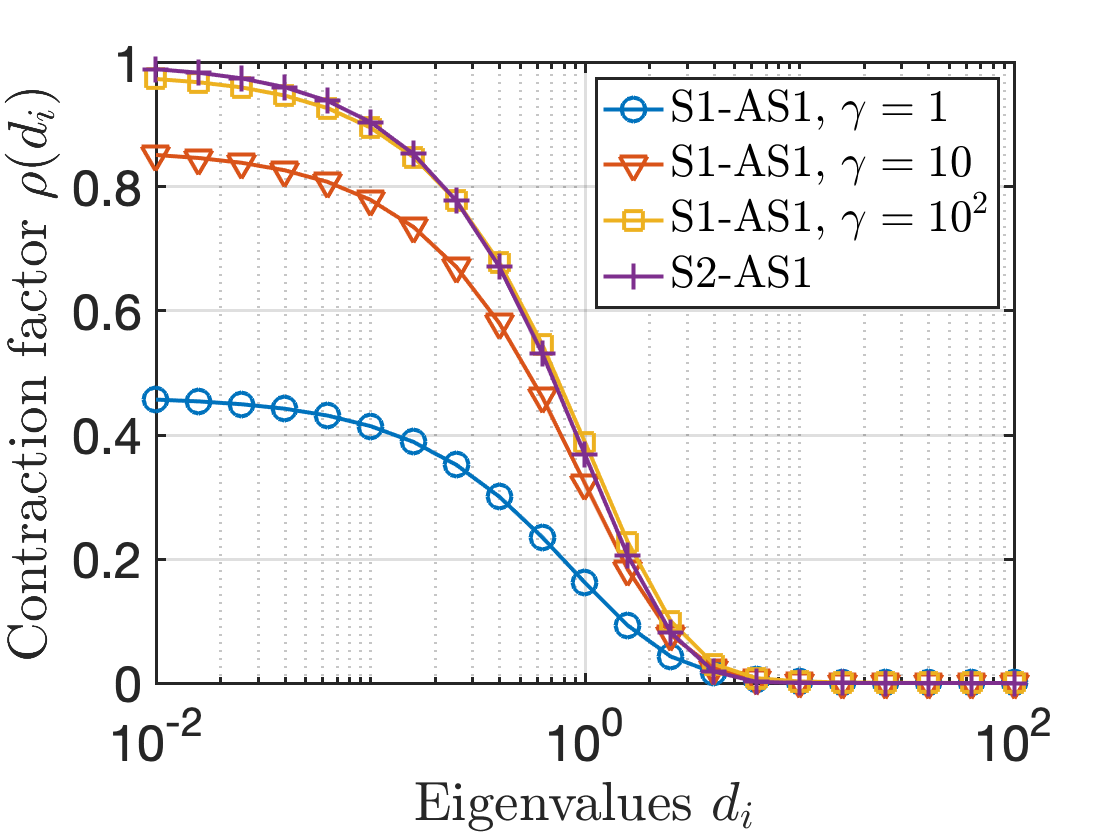}
\includegraphics[scale = 0.13]{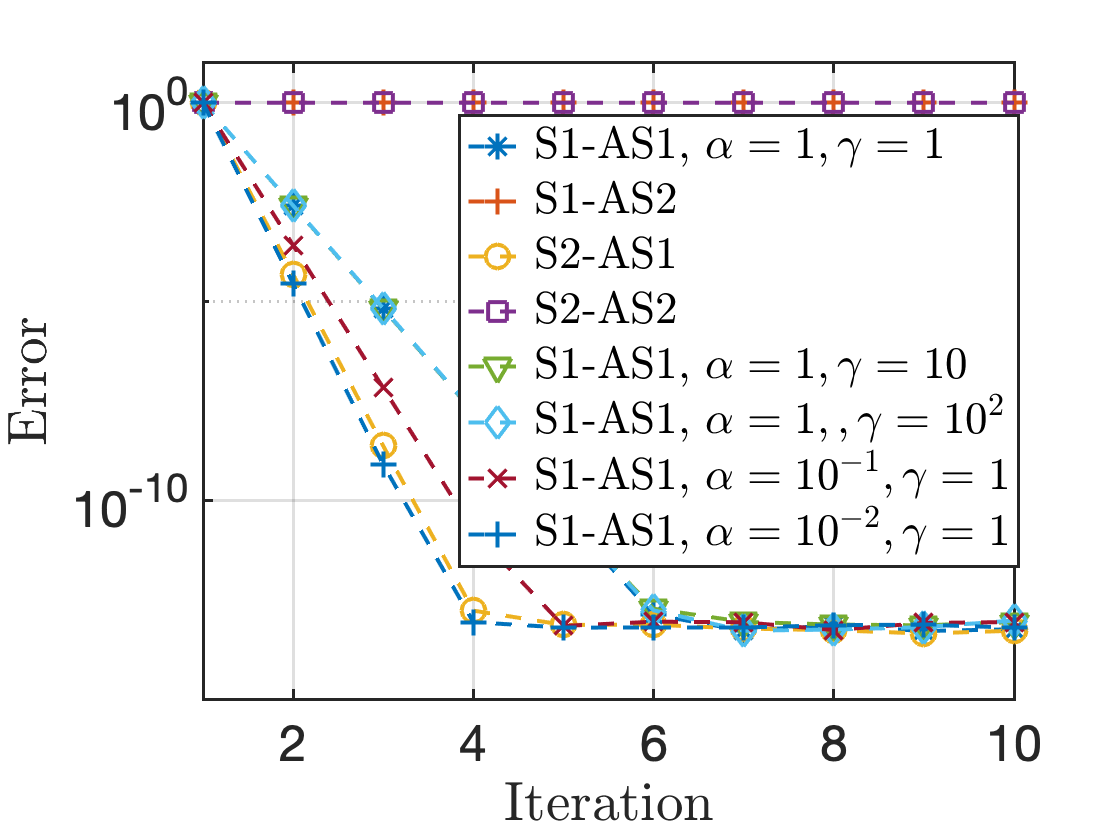}
\includegraphics[scale = 0.13]{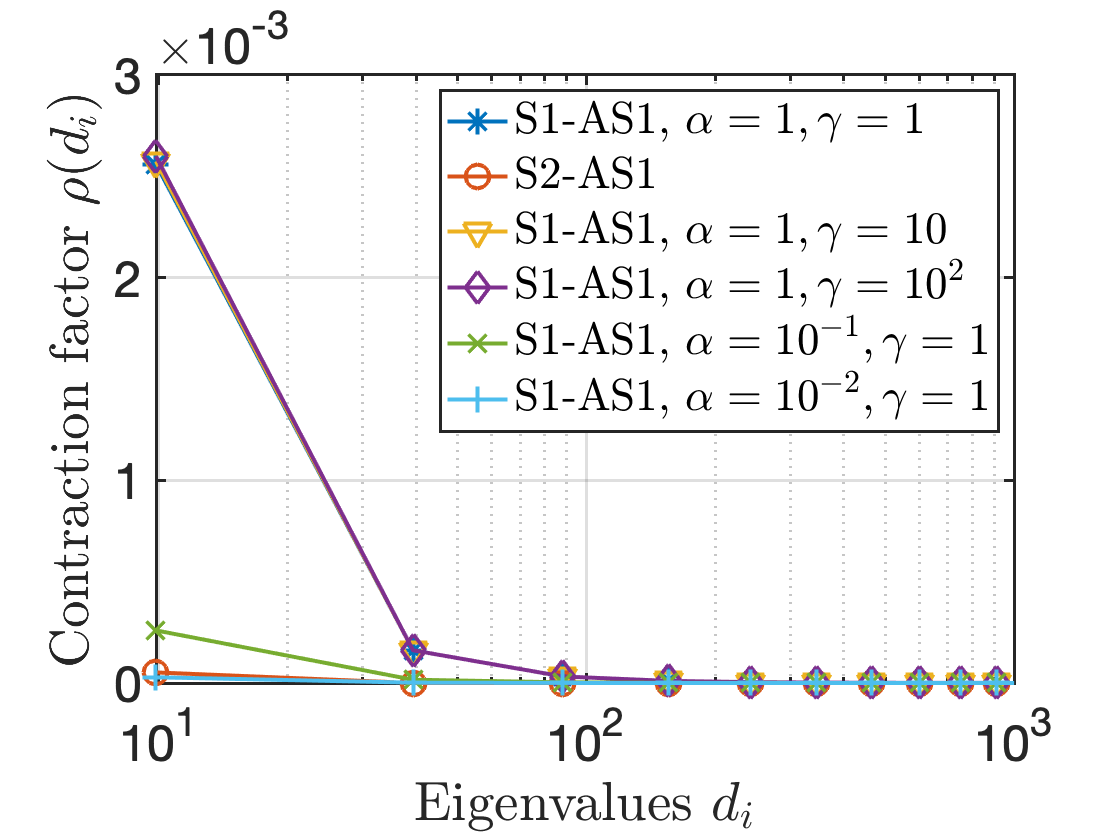}
\caption{Top left: contraction factors~\eqref{eq:rhoS1AS1} and~\eqref{eq:rhoS2AS1} as functions of eigenvalues $d_i\in[10^{-2}, 10^2]$ with different values of $\gamma$. Top right: error decay of four algorithms~\eqref{eq:S1AS1}-\eqref{eq:S1AS2}. Bottom: contraction factors for eigenvalues $d_i\in[10, 10^3]$ with different values of $\gamma$ and $\alpha$}
\label{fig:illustration}
\end{figure}
As expected, increasing the value of $\gamma$ in~\eqref{eq:rhoS1AS1} leads to behavior more similar to that of~\eqref{eq:rhoS2AS1} for all eigenvalues. We also observe a slight difference in the range of $d_i=1$ between two curves $\rho^{\text{AS1}}_{\text{S2}}$ and $\rho^{\text{AS1}}_{\text{S1}}$ with $\gamma = 10^2$, due to the value $\alpha = 1$ in~\eqref{eq:rhoS1AS1}. Taking $\alpha = 0$, the curve of $\rho^{\text{AS1}}_{\text{S1}}|_{\alpha = 0}$ is then always below that of $\rho^{\text{AS1}}_{\text{S2}}$, since the second fraction in~\eqref{eq:rhoS1AS1alpha0} is always smaller than one. More generally, we observe that all convergence factors are monotonically decreasing, and their values are especially small for large eigenvalues. This indicates that both algorithms~\eqref{eq:S1AS1} and~\eqref{eq:S2AS1} are good smoothers, and their convergence behavior for small eigenvalues can be further improved by introducing a relaxation parameter.

Next, we test the error decay of four algorithms~\eqref{eq:S1AS1}-\eqref{eq:S1AS2} in the one-dimensional case. We keep the same setting as in the previous test and consider the length $L=1$, the initial condition $y_0(x) = \frac{1}{2\pi^2}(1-\exp(\pi^2T))\sin(\pi x)$, the target function $\hat y(t, x) = \frac{1}{2\pi^2}(\exp(\pi^2t)-\exp(\pi^2T))\sin(\pi x)$ in~\eqref{eq:S1}, and the final condition $y_T(x) = 0$ in~\eqref{eq:S2}. In particular, this target function is also the solution of the optimality system~\eqref{eq:S2} with $\lambda(t, x) = (\exp(\pi^2t) - \exp(\pi^2T)/2)\sin(\pi x)$. We use the Crank-Nicolson scheme to discretize both systems with mesh size $h_t = h_t = 1/32$. Fig.~\ref{fig:illustration} on the top right illustrates the error decay of four algorithms. As explained in our theoretical analysis, both algorithms~\eqref{eq:S2AS2} and~\eqref{eq:S1AS2} stagnate as $\rho^{\text{AS2}}_{\text{S1}} = \rho^{\text{AS2}}_{\text{S2}}$ = 1. Regarding algorithms AS1, we observe that the algorithm~\eqref{eq:S2AS1} converges faster than~\eqref{eq:S1AS1} in the case $\alpha = 1$ and $\gamma = 1$. However, when increasing the value of $\gamma$, the convergence of~\eqref{eq:S1AS1} stays the same. Indeed, the smallest eigenvalue in this case is approximately $\pi^2$, and Fig.~\ref{fig:illustration} on the bottom further illustrates contraction factors for large eigenvalues. We observe that $\rho^{\text{AS1}}_{\text{S2}}$ is much smaller in this case, and increasing $\gamma$ in $\rho^{\text{AS1}}_{\text{S1}}$ cannot track the behavior of $\rho^{\text{AS1}}_{\text{S2}}$. However, decreasing $\alpha$ in $\rho^{\text{AS1}}_{\text{S1}}$ can better track the convergence behavior of $\rho^{\text{AS1}}_{\text{S2}}$, which has also been testified in Fig.~\ref{fig:illustration} on the top right.

\bibliography{sample}
\bibliographystyle{spmpsci}

\end{document}